\documentclass[preprint,12pt]{elsarticle}

\usepackage{comment,bm}
\usepackage{amssymb,amsmath}
\usepackage{amsthm}
\usepackage[all]{xy}

\newdir{ >}{{}*!/-5pt/@{>}}

\newtheorem{theorem}{Theorem}[section]
\newtheorem{corollary}[theorem]{Corollary}
\newtheorem{lemma}[theorem]{Lemma}
\newtheorem{proposition}[theorem]{Proposition}

\theoremstyle{definition}
\newtheorem{definition}[theorem]{Definition}
\newtheorem{remark}[theorem]{Remark}
\newtheorem{example}[theorem]{Example}

\newcommand{\Cone}{\mathrm{Cone}}
\newcommand{\coker}{\mathrm{coker}}
\renewcommand{\ker}{\mathrm{ker}}
\renewcommand{\Im}{\mathrm{Im}}

\renewcommand{\H}{\mathrm{H}} 
\newcommand{\Hl}{\mathrm{H}_{\lambda}} 

\newcommand{\Id}{\mathrm{Id}} 

\newcommand{\Hom}{\mathrm{Hom}} 
\newcommand{\HomA}{\mathrm{Hom}_{\calA}} 

\newcommand{\calA}{\mathcal{A}} 

\newcommand{\CA}{\mathbf{C}(\calA)}

\newcommand{\CAast}[1]{\mathbf{C}^{\ast}(\calA)}

\newcommand{\KA}{\mathbf{K}(\calA)}
\newcommand{\KAs}[1]{\mathbf{K}^{#1}(\calA)}  
\newcommand{\KAast}[1]{\mathbf{K}^{\ast}(\calA)}

\newcommand{\DA}{\mathbf{D}(\calA)} 
\newcommand{\DlA}{\mathbf{D}_{\lambda}(\calA)} 
\newcommand{\DlAs}[1]{\mathbf{D}_{\lambda}^{#1}(\calA)}
\newcommand{\DlAast}[1]{\mathbf{D}_{\lambda}^{\ast}(\calA)}

\newcommand{\DlsgA}{\D_{\lambda\textrm{-}sg}^{b}({\cal A})}

\newcommand{\C}{\mathbf{C}}
\newcommand{\D}{\mathbf{D}}
\newcommand{\K}{\mathbf{K}}

\newcommand{\f}{\bm{f}}
\newcommand{\g}{\bm{g}}
\newcommand{\h}{\bm{h}}
\newcommand{\s}{\bm{s}}
\newcommand{\uu}{\bm{u}}
\newcommand{\vv}{\bm{v}}
\renewcommand{\t}{\bm{t}}

\renewcommand{\a}{\bm{a}}
\renewcommand{\b}{\bm{b}}

\newcommand{\X}{{\bm{X}}}
\newcommand{\Y}{{\bm{Y}}}
\newcommand{\Z}{{\bm{Z}}}
\newcommand{\p}{{\bm{P}}}
\renewcommand{\P}{{\bm{P}}}
\newcommand{\Q}{{\bm{Q}}}

\newcommand{\lPP}{\mathcal{PP}_\lambda}

\newcommand{\GlPP}{\mathcal{G}(\mathcal{PP}_\lambda)}

\newcommand{\infl}{\inf_{\lambda}}
\newcommand{\supl}{\sup_{\lambda}}
\newcommand{\plext}{\rm{P}_{\lambda}\rm{ext}}
\newcommand{\plpdA}{\mathrm{p}_{\lambda}.\mathrm{pd}_{\cal A}}
\newcommand{\plgldA}{\mathrm{P}_{\lambda}.\mathrm{gldim}{\cal A}}

\newcommand{\ra}{\rightarrow}
\newcommand{\Ra}{\Rightarrow}

\journal{Chinese Annals of Mathematics, Series B}

\begin{document}

\begin{frontmatter}

\title{$\lambda$-pure global dimension of Grothendieck categories and some applications}

\author[label1,label2]{Xi Wang }
\address[label1]{School of Mathematics, Statistics and Mechanics, Beijing University of Technology, Beijing 100124, China}
\address[label2]{College of Mathematics, Sichuan University of Arts and Sciences,Dazhou 635000,China}

\cortext[cor1]{Corresponding author: Hailou Yao }

\ead{xwang1233@163.com}

\author[label1]{Hailou Yao}

\ead{yaohl@bjut.edu.cn}

\author[label2]{Lei Shen}
\ead{shenleisasu@163.com}

\begin{abstract}
We study the $\lambda$-pure global dimension of a Grothendieck category $\cal A$,
and provide two different applications about this dimension.
We obtain that if the $\lambda$-pure global dimension $\plgldA<\infty$, then
(1) The ordinary bounded derived category (where $\cal A$ has enough projective objects) and the bounded $\lambda$-pure one differ only by a homotopy category;
(2) The $\lambda$-pure singularity category $\DlsgA =0$.
At last, we explore the reason why the general construction of classic Buchweitz-Happel Theorem is not feasible for $\lambda$-pure one.
\end{abstract}


\end{frontmatter}

\section{Introduction}

The $\lambda$-purity and $\lambda$-pure derived category have been widely explored in the literatures, see for instance Neeman \cite{Neeman1990}, Ad\'{a}mek-Rosicky \cite{Adamek1994, Adamek2004}, Izurdiaga \cite{Izurdiaga2021}, and Shen \emph{et.al.} \cite{Shen2022}.
As a continuation of the work in \cite{Wang}, we consider the $\lambda$-pure projective dimension and $\lambda$-pure global dimension, and find some interesting applications.

Specifically, we first introduce some necessary properties about $\lambda$-purity in Section 2, and also give the definitions of $\lambda$-pure projective dimension and the $\lambda$-pure global dimension of a Grothendieck category $\cal A$.
By definition, the complex $\X\in\DlA$ is of $\lambda$-pure projective dimension at most $n$, written by $\plpdA \X\leq n$, if there exists a $\lambda$-pure projective resolution $\P\ra \X$ with $P^i=0$ for any $i<-n$;
the $\lambda$-pure global dimension of $\cal A$,  written by $\plgldA$, is the supremum of the $\lambda$-pure projective dimension of all objects in $\cal A$.
Some similar dimension definitions can be found in literatures Huang\cite{Huang2013}, Li-Huang\cite{Lihuanhuan2015} and Zheng-Huang\cite{HZY2016}.

Neeman\cite{Neeman1990} investigated the derived category of Quillen's exact categories (see B\"{u}hler\cite{Theo2010} and Keller\cite[Appendix A]{Keller1990}) in 1990,
which is the Verdier quotient of a homotopy category by a thick subcategory.
By \cite{Wang}, let $\K_{\lambda}{(\cal A)}$ be a subcategory of $\KA$ consisting of all $\lambda$-pure acyclic complexes, then it is thick.
Naturally, the corresponding Verdier quotient is the $\lambda$-pure derived category of $\cal A$, that is, $\DlA:= \KA/\K_{\lambda}{(\cal A)}$.
Moreover, $\DlAs{\ast}:= \KAs{\ast}/\K_{\lambda}^{\ast}{(\cal A)}$, where $\ast \in \{ b, -, +\}$.
In Section 3, we suppose the Grothendieck category $\cal A$ have enough projective objects.
Then we compare the ordinary  derived category and the $\lambda$-pure one,
and we get the following results.

Let $\cal A$ be a Grothendieck category with enough projective objects.
We have a triangle equivalence
$$\D^-({\cal A})\simeq \D_{\lambda}^{-}{(\cal A)}/\K_{ac}^{-}({\lPP}).
$$.

For the case of bounded one,
we obtain that if $\plgldA$ is finite, then there exists a triangle equivalence
$$\D^{b}({\cal A}) \simeq \D_{\lambda}^{b}{(\cal A)}/\K_{ac}^{b}(\lPP),
$$
where $\K_{ac}^{b}(\lPP)$ is the subcategory of $\K^b(\cal A)$ consisting of all acyclic cochain complexes of $\lambda$-pure projective objects.

Assume that $\cal A$ is an abelian category with enough projective objects.
Denoted by $\K^{b}(\cal P)$ the subcategory of $\K^b(\cal A)$ consisting of all complexes of projective objects.
Then by \cite[Lemma 5.1.10]{ZP2015}, it is a triangulated subcategory of the bounded derived category $\D^{b}({\cal A})$.
And it is also a thick one by Buchweitz.
Thus, the Verdier quotient category $\D^{b}({\cal A})/\K^{b}(\cal P)$, in fact, is the usual singularity category of the abelian category $\cal A$, denoted by $\D_{sg}^{b}(\cal A)$.
Specially, let $R$ be a ring, ${\cal A}=R\textrm{-Mod}$, then $\D_{sg}^{b}{(R)}=0$ if and only if the projective dimension of any $R$-module is finite.
More generally, Christensen \emph{et.al.} \cite{Christensen2023} explored the singularity category of an exact category, where the exact structure they chosen is ordinary one, and they generalized the above result to the exact category.
The relative version of singularity categories was studied by Chen \cite{Chen2011}, Li-Huang \cite{Lihuanhuan2015, Lihuanhuan2020}, respectively.
Here we concern the $\lambda$-pure version in Section 4, that is, the $\lambda$-pure singularity category of the Grothendieck category $\cal A$, denoted by ${\DlsgA} :=\DlAs b/\K^{b}(\lPP)$.
And we also get the result that ${\DlsgA}=0$ if and only if the $\lambda$-pure global dimension $\plgldA$ is finite.

The classical  Buchweitz-Happel Theorem (and its inverse) says that the singularity category is equivalent to a stable category under some conditions, see Theorem \ref{theoremB-H} for more details.
Moreover,
Gao-Zhang\cite{GZ2010}, Chen\cite{Chen2011}, and Li-Huang\cite{Lihuanhuan2015, Lihuanhuan2020} got the relative version of Buchweitz-Happel Theorem, respectively.
However, we fail to get the $\lambda$-pure one.
We state the reason why the general construction of classic Buchweitz-Happel Theorem is not feasible for $\lambda$-pure version at the end of Section 4.

\section{$\lambda$-pure projective dimension}

We first explain some necessary notions which will be used repeatedly in this manuscript. And more details can be found in \cite{Wang}, because this manuscript is its continuation.

In what follows, $\cal A$ denotes a Grothendieck category, i.e., $\cal A$ is an abelian category having all coproducts and a generator such that exact directed colimits exist.
And $\lambda$ denotes an infinite regular cardinal, i.e., it is not a sum of a smaller number of smaller cardinals. In other words, the cardinal $\lambda$ satisfies: for any family $(\mu_\alpha)_{\alpha<\mu}$ of cardinals, where $\mu<\lambda$ and  $\mu_\alpha<\lambda$ for each $\alpha<\mu$, one always has $\sup_{\alpha<\mu} \mu_\alpha=\bigcup_{\alpha<\mu} \mu_\alpha<\lambda$.
The least infinite cardinal is $\omega=\aleph_0$.

By \cite[Definition1.13(1)]{Adamek1994}, a poset is called $\lambda$-directed provided that every subset of cardinality smaller than $\lambda$ has an upper bound.
A diagram whose scheme is a $\lambda$-directed poset is called a $\lambda$-directed diagram, and its colimit is called a $\lambda$-directed colimit.

An object $F$ in a category  $\cal A$ is $\lambda$-presentable if the functor
$\Hom_{\cal A}(F,-)$ commutes with $\lambda$-directed colimits in $\cal A$.

A short exact sequence $\xi :\xymatrix@C=0.5cm{
  0 \ar[r] & A \ar[r]^{f} & B \ar[r]^{g} & C \ar[r] & 0 }$
is \emph{$\lambda$-pure}, if for any $\lambda$-presentable object $F$, $\HomA(F, \xi)$ is an exact sequence.

An object $X$ in $\calA$ is called $\lambda$-pure projective, if $X$ is projective with respect to $\lambda$-pure exact sequence. And $\lPP(\calA)$, or simply $\lPP$, denotes the class of all $\lambda$-pure projective objects.
Note that the $\omega$-purity is just the classical purity.

Let $\X, \Y$ be two complexes in $\CA$, by definition, the \emph{total complex} $\Hom_{\cal A}(\X,\Y)$ is a complex with
the component
$$\Hom_{\cal A}(\X,\Y)^n :=\prod_{i\in \mathbb{Z}}\Hom_{\cal A}(X^i,Y^{i+n})$$
and the differential $d^n$ is defined by $$d^n(f)=(d_{\Y}^{n+i}f^i+(-1)^{n+1}f^{i+1}d_{\X}^{i})_{i\in \mathbb{Z}}$$
for any $f\in \Hom_{\cal A}(\X,\Y)^n$.
A useful formula, which connects the homology and homotopy, is $\H^n \Hom_{\cal A}(\X,\Y)= \Hom_{\KA}(\X,\Y[n])$ with $n\in \mathbb{Z}$ for any $\X, \Y\in \CA $.

A complex $\X\in \CA$ is $\lambda$-pure acyclic if and only if the complex $\Hom_{\cal A}(P, \X)$ is acyclic (or exact) for any $\lambda$-presentation (more generally, $\lambda$-pure projective) object $P$.

Put $\K_{\lambda}{(\cal A)}:=\{\X\in \KA| \X\textrm{ is }\lambda\textrm{-pure acyclic}\}$.
Then one can verify that $\K_{\lambda}{(\cal A)}$ is a thick subcategory of $\KA$.
The $\lambda$-pure derived category of $\cal A$ defined by a Verdier quotient category
$$\DlA: = \KA/\K_{\lambda}(\cal A),$$
see \cite{Wang}.
Moreover,
$$\DlAs{\ast}:= \KAs{\ast}/\K_{\lambda}^{*}(\cal A),$$
where $\ast \in \{ b, -, +\}$.

\begin{remark}
Note that Gillespie \cite{Gillespie2016} considered the derived category with respect to a generator; as an application, they presented the $\lambda$-pure derived category in \cite[Section 4.6]{Gillespie2016}.
In fact, Gillespie got the $\lambda$-pure derived category by a special model structure, which is essentially the same as our definition.
However, Gillespie's attention is focused on the construction of model structure, rather than the structure of the $\lambda$-pure derived category itself, which is also the main purpose of our manuscript.
\end{remark}

A cochain map $\f: \X\ra \Y$ in $\CA$ is called a $\lambda$-pure quasi-isomorphism if the mapping cone $\Cone (\f)$ is $\lambda$-pure acyclic.

For any $\X\in \CA$, a $\lambda$-pure quasi-isomorphism $\f: \p\ra \X$ is called a $\lambda$-pure projective resolution of $\X$, if $\p$ is a  complex of $\lambda$-pure projective objects.

Set
\begin{align*}\K_{\lambda}^{-, b}{(\cal X)}
&=\{\X \in \K^{-}{(\cal X)} \mid \exists~ n(\X)\in \mathbb{Z},\textrm{ such that }\Hl^{i}(\X)=0 \\
&\textrm{~~for all~~} i\leq n(\X)  \},
\end{align*}
where $\Hl^{i}(\X)=0$ for all $i\leq n(\X)$ means the complex $\X$ is $\lambda$-pure acyclic in degree $\leq n(\X)$.
In fact,
\begin{align*}
\X\in \K_{\lambda}^{-, b}(\cal X) &\Longleftrightarrow \textrm{ for any } \lambda\textrm{-presentable (or } \lambda \textrm{-projective) object }F,  ~\\
&\Hom_{\cal A}(F,\X)\textrm{ is acyclic in degree } \leq n(\X);\\
&\Longleftrightarrow \textrm{ for any } \lambda\textrm{-presentable (or } \lambda \textrm{-projective) object }F,  ~\\
&\H^i\Hom_{\cal A}(F,\X)=0 \textrm{ for all } i\leq n(\X).
\end{align*}


Set
\begin{center}
$\supl\X:=\sup\{i\in\mathbb{Z}, \H_{\lambda}^{i}(\X)\neq 0
\},$
\end{center}
and
\begin{center}
$\infl\X:=\inf\{i\in\mathbb{Z}, \H_{\lambda}^{i}(\X)\neq 0
\}.
$
\end{center}

\begin{remark}
If $\X\ra \Y$ is a $\lambda$-pure quasi-isomorphism, then $\supl\X=\supl\Y$ and $\infl\X=\infl\Y$, which follows by direct calculating.
\end{remark}

\begin{proposition}\label{contractible}
For any object $X$ in $\cal A$,
assume the complex
$$\xymatrix@C=0.5cm{
 \P: &\cdots  \ar[r] & P_1 \ar[r]^{d_1} & P_0 \ar[r]^{d_0} & X \ar[r] & 0 }
$$
is a $\lambda$-pure acyclic complex with $P_i$ $(i=0,1,\cdots)$ $\lambda$-pure projective.
Then we have that $X$ is $\lambda$-pure projective if and only if $\P$ is contractible.
\end{proposition}
\begin{proof}
Suppose $\P$ is contractible, it is clear that $X$ is a direct summand of $P_0$. Then $X$ is $\lambda$-pure projective.

Consider the following diagram
$$\xymatrix@C=0.5cm{
  &\cdots \ar[r]
&
  P_1
   \ar[rr]^{d_1} \ar[d]_{\Id_{P_1}}
&&
  P_0
   \ar[rr]^{d_0} \ar[d]^{\Id_{P_0}}\ar@{-->}[dll]_{t_{0}} \ar@{-->}[ddl]_(0.4){j_{0}}
&&
  X
  \ar[r]  \ar[d]^{\Id_X} \ar@{-->}[dll]^{t_{-1}}
&
0 \\
  &\cdots \ar[r]
&
  P_1
   \ar[rr]_{d_1} \ar@{->>}[dr]_{\pi _1}
&&
  P_0
   \ar[rr]_{d_0}
&&
  X
  \ar[r]
&
0 \\
&&&
K_0\ar@{ >->}[ur]_{i_0}
&&&&
  }
$$
where $i_0=\ker d_0$.
Since $X\in \lPP$, there exists a morphism $t_{-1}:X\ra P_0$ such that $\Id_X=d_0t_{-1}$.
Then by the equivalences $d_0\Id_{P_0}=\Id_Xd_0=d_0t_{-1}d_0$, we have $d_0(\Id_{P_0}-t_{-1}d_0)=0$, that is, there exists a morphism $j_0:P_0\ra K_0$ such that $i_0j_0=\Id_{P_0}-t_{-1}d_0$.
Since $P_0\in \lPP$ and $\pi_{1}$ is a $\lambda$-pure epimorphism, we have $j_0=\pi_1t_0$ with $t_0:P_0\ra P_1$.
Hence $\Id_{P_0}=t_{-1}d_0+i_0\pi_1t_0=t_{-1}d_0+d_1t_0$.
Denote $X=P_{-1}$, $t_{-2}=0$. Then by induction, there exists $t_{i}:P_i\ra P_{i-1}$, such that $\Id_{P_i}=t_{i-1}d_i+d_{i+1}t_i$, $i=-1,0,1,2,\cdots$, which means $\P$ is contractible.
\end{proof}

\begin{corollary}
Let $\P\in \K^{-}(\lPP)$. If $\P$ is $\lambda$-pure acyclic, then $\P=0$ in $\KA$.
\end{corollary}

\begin{proof}
Without loss of generality, we set
$$\xymatrix@C=0.6cm{
  \P: & \cdots \ar[r] & P_{n} \ar[r]^{d_n} & \cdots \ar[r] & P_{2} \ar[r]^{~d_{2}} & P_{1}\ar[r]^{~d_{1}} & P_{0}\ar[r] & 0 .}
$$
Then by the above Proposition we have $\P$ is contractible, that is, $\P=0$ in $\KA$.
\end{proof}

More generally, we have the following result.

\begin{proposition}\label{prop204}
Let $\P\in \K(\lPP)$. If $\P$ is $\lambda$-pure acyclic, then $\P=0$ in $\KA$.
\end{proposition}

\begin{proof}
By \cite[Theorem 1.3]{Shen2022}, we know $\Hom_{\KA}{(\K(\lPP),\K_{\lambda}(\cal A))}=0$.
It is clear that the above $\P$ belongs to both $\K(\lPP)$ and $\K_{\lambda}(\cal A)$.
Therefore, we have $\Id_{\P}=0$, that is, $\P=0$ in $\KA$.
\end{proof}

\begin{definition}\label{def206}
The complex $\X\in\DlA$ is said to be of \emph{$\lambda$-pure projective dimension} at most $n$, written by $\plpdA \X\leq n$, if there exists a $\lambda$-pure projective resolution $\P\ra \X$ with $P^i=0$ for any $i<-n$.
If $\plpdA \X\leq n$ for all $n$, then we write $\plpdA\X=-\infty$, in this case, $\X$ is a $\lambda$-pure acyclic complex.
If there is no $n$ such that $\plpdA\X\leq n$, then we write $\plpdA\X=\infty$.
\end{definition}

The $\lambda$-pure projective dimension can be expressed as
\begin{align*}
\plpdA\X= &-\sup\{\inf\{n\in \mathbb{Z}|P^n\neq 0
\}|\P\ra \X\textrm{ is a } \lambda \textrm{-pure projective }\\
& \textrm{ resolution of } \X\}
\end{align*}

\begin{remark}
For an object $X\in \cal A$, which can be viewed as a stalk complex concentrated in degree 0, $\plpdA X\leq n$ means that there exists a $\lambda$-pure exact sequence
$$
0\ra P_n \ra \cdots \ra P_1\ra P_0\ra X\ra 0
$$
with $P_i (i=0,1,\cdots,n)$ $\lambda$-pure projective.
\end{remark}

\begin{definition}\label{def208}
The \emph{$\lambda$-pure global dimension} of $\cal A$, written by $\plgldA$, is the supremum of the $\lambda$-pure projective dimensions of all objects in $\cal A$.
It can be expressed as
$$\plgldA=\sup\{\plpdA X |X\in \cal A
\}$$
\end{definition}

Let $\P\ra X$ be a $\lambda$-pure projective resolution of $X$.
For each integer $i$ and each object $Y$, we define a\emph{ $\lambda$-pure derived functor}, denoted by $\plext_{\cal A}^{i}(X,Y)$, as
$$\plext_{\cal A}^{i}(X,Y):=\H^{i}\Hom_{\cal A}(\textbf{P},Y).$$

\begin{lemma}\label{lem207}
(1)\ \ Let $\xymatrix@C=0.2cm{
   \X\ar[rr] && \Y }$ be a $\lambda$-pure quasi-isomorphism in $\C{(\calA)}$ with $\X \in \K^b{(\calA)}, \Y \in \K^-(\calA)$, then there exists a $\lambda$-pure quasi-isomorphism $\xymatrix@C=0.2cm{
   \Y\ar[rr] && \Z }$ with $\Z \in \K^b(\calA).$

(2)\ \  Let $\xymatrix@C=0.2cm{
   \Y \ar[rr] && \Z }$ be a $\lambda$-pure quasi-isomorphism in $\C{(\calA)}$ with $\Z \in \K^{-}{(\calA)}, \Y \in \K(\calA)$, then there exists a $\lambda$-pure quasi-isomorphism $\xymatrix@C=0.2cm{
   \X\ar[rr] && \Y }$ with $\X \in \K^-(\calA).$

\end{lemma}

\begin{proof}
(1)\ \ Without loss of generality, assume that $Y^i=0$ for all $i>0$, $X^j= 0$ for all $j\leq -m$.
Since $\f: \X\ra \Y$ is a $\lambda$-pure quasi-isomorphism, we have $\Hom_{\calA}(P, \f)$ is a quasi-isomorphism for any $\lambda$-pure projective object $P$.
This implies $\H^i\Hom _{\calA}(P,\f)$ is an isomorphism for any $i\in \mathbb{Z}$,
so we can suppose $\H^j\Hom _{\calA}(P,\Y)=0$ for all $j\leq -m$, that is, $\Y$ is $\lambda$-pure acyclic in degree $\leq -m$.
Consider the following commutative diagram
$$\xymatrix@C=0.5cm{
 \cdots\ar[r]& Y^{-m-1} \ar[d] \ar[r] & Y^{-m}\ar[d] \ar[r] & Y^{-m+1}\ar@{=}[d] \ar[r]& \cdots \ar[r]& Y^{0} \ar@{=}[d]\ar[r] & 0 \ar[r]& \cdots
  \\
 \cdots\ar[r]& 0 \ar[r] & \ker d_Y^{-m+1} \ar[r] & Y^{-m+1} \ar[r]& \cdots \ar[r]& Y^{0} \ar[r] & 0 \ar[r]& \cdots }
$$
Let $\Z$ denote the complex as shown in the second row of the above diagram, hence $\Y\ra \Z$ is a $\lambda$-pure quasi-isomorphism as wanted.

(2)\ \ It is similar to assertion (1), and we can suppose $\H^j\Hom _{\calA}(P,\Y)=0$ for all $j\geq 0$, where $P$ is a $\lambda$-pure projective object. Then consider the following commutative diagram
$$\xymatrix@C=0.5cm{
 \cdots\ar[r]& Y^{-1} \ar@{=}[d] \ar[r] & \textrm{Im}d_{Y}^{-1}\ar[d] \ar[r] & 0\ar[d] \ar[r]& \cdots
  \\
 \cdots\ar[r]& Y^{-1} \ar[r]& Y^{0} \ar[r]& Y^{1} \ar[r]& \cdots }
$$
The first row of the above diagram is complex $\X$ as expected, then $\X\ra \Y$ is a $\lambda$-pure quasi-isomorphism.
\end{proof}

According to Lemma \ref{lem207} and \cite[Proposition III.2.10]{Gelfand2003}, we immediately obtain the following proposition.

\begin{proposition}\label{prop207}
Let $\calA$ be a locally $\lambda$-presentable Grothendieck category. Then we have

(1)\ \ $\D_{\lambda}^{b}({\calA})$ is a full subcategory of $\D_{\lambda}^{-}({\calA})$, and $\D_{\lambda}^{-}({\calA})$ is a full subcategory of $\D_{\lambda}({\calA})$.

(2)\ \ $\D_{\lambda}^{b}({\calA})$ is a full subcategory of $\D_{\lambda}^{+}({\calA})$, and $\D_{\lambda}^{+}({\calA})$ is a full subcategory of $\D_{\lambda}({\calA})$.
\end{proposition}

\begin{corollary}
Let $\calA$ be a locally $\lambda$-presentable Grothendieck category, then $\D_{\lambda}^{b}({\calA})$, $\D_{\lambda}^{-}({\calA})$, $\D_{\lambda}^{+}({\calA})$ are the triangle subcategories of $\D_{\lambda}({\calA})$, and
$$\D_{\lambda}^{+}({\calA}) \cap \D_{\lambda}^{-}({\calA}) = \D_{\lambda}^{b}({\calA}).$$
\end{corollary}

\begin{theorem}
Let $X, Y \in \cal A$. Then we have $${\plext}_{\cal A}^{n}(X,Y)\cong \Hom_{\D_{\lambda}^b(\cal A)}(X,Y[n]).$$
\end{theorem}

\begin{proof}
Let $\P\ra X$ be a $\lambda$-pure projective resolution of $X$.
Then $\P\ra X$ is a $\lambda$-pure quasi-isomorphism, and $\P\cong X$ in ${\D_{\lambda}^b(\cal A)}$.
Hence we have the equivalences as follows
\begin{align*}
{\plext}_{\cal A}^{n}(X,Y)&= \H^{n}\Hom_{\cal A}(\textbf{P},Y)\\
& = \Hom_{\KA}(\P,Y[n])\\
&\cong  \Hom_{\DlA}(\P,Y[n])\\
&\cong  \Hom_{\D_{\lambda}^b(\cal A)}(X,Y[n]),
\end{align*}
where the first isomorphism holds from \cite[Lemma 4.8]{Wang}, and the second isomorphism holds from Proposition \ref{prop207}(1).
\end{proof}

\begin{proposition}\label{prop211}
$\K_{\lambda}^{-,b}(\lPP)$ is a thick subcategory of $\K^{-}(\lPP)$.
\end{proposition}

\begin{proof}
We know $\K^{-,b}_{\lambda}(\lPP)$ is a full subcategory of $\K^{-}(\lPP)$. By \cite[Lemma 3.7]{Wang}, $\K^{-,b}_{\lambda}(\lPP)$ is also closed under direct summands. Hence if $\K^{-,b}_{\lambda}(\lPP)$ is a triangulated subcategory of $\K^{-}(\lPP)$, then it is also a thick subcategory.

In fact, let $\X\ra \Y\ra \Z\ra \X[1]$ be a distinguished triangle in $\K^{-}(\lPP)$ with $\X,\Z\in \K^{-,b}_{\lambda}(\lPP)$.
Then there exist integers $n_1, n_2$, such that $\Hl^{i}(\X)=0$ and $\Hl^{j}(\Z)=0$ for all $i\leq n_1$ and $j\leq n_2$.
Set $n=\min\{n_1, n_2\}$.
Consider the sequence
$$\Hl^{t}(\X)\ra \Hl^{t}(\Y)\ra \Hl^{t}(\Z),
$$
we have $\Hl^{t}(\Y)=0$ for all $t\leq n$.
That means $\Y\in \K^{-,b}_{\lambda}(\lPP)$, and $\K^{-,b}_{\lambda}(\lPP)$ is closed under extension.

Therefore, $\K^{-,b}_{\lambda}(\lPP)$ is a triangulated subcategory of $\K^{-}(\lPP)$ as wanted.
\end{proof}

In \cite[Corollary 4.3]{Wang} we know the composition of two $\lambda$-pure quasi-isomorphisms is also a $\lambda$-pure quasi-isomorphism. Moreover, we have the ``two out of three" property.

\begin{proposition}\label{prop206}
For two composable cochain map $\xymatrix@C=0.5cm{
  \X \ar[r]^{\f} & \Y \ar[r]^{\g} & \Z }$, if two of the three morphisms $\f, \g, \g\f$ are $\lambda$-pure quasi-isomorphisms, then so is the third one.
\end{proposition}
\begin{proof}
We only prove the case of which $\g$ and $\g\f$ are $\lambda$-pure quasi-isomorphisms.

For any $\lambda$-pure projective object $P$, by \cite[Proposition 4.2]{Wang}, we know $\Hom_{\K(\cal A)}(P,\X)\cong  \Hom_{\K(\cal A)}(P,\Z)\cong  \Hom_{\K(\cal A)}(P,\Y)$, which implies that $\f: \X\ra \Y$ is a $\lambda$-pure quasi-isomorphism.
\end{proof}

\begin{proposition}
For any complex $\X\in \D_{\lambda}^{b}(\cal A)$, and $n\in \mathbb{Z}$, the following statements are equivalent.

(1)\ \ $\plpdA\X\leq n$.

(2)\ \ $\infl \X \geq -n$, and if $\f^{'}:\P^{'}\ra \X$ is a $\lambda$-pure projective resolution of $\X$, then $\coker d_{\P^{'}}^{-n-1}$ is $\lambda$-pure projective.

(3)\ \ If $\f^{'}:\P^{'}\ra \X$ is a $\lambda$-pure projective resolution of $\X$, then $\P^{'}=\P_1\bigoplus \P_2$, where $P_1^i=0$ for any $i<-n$, and $\P_2$ is contractible.
\end{proposition}

\begin{proof}
(1) $\Ra$ (2)\ \ Let $\plpdA\X\leq n$, i.e., there exists a $\lambda$-pure projective resolution $\f: \P\ra \X$ with $P^{i}=0$ for all $i<-n$.
Obviously, we have $\infl \X =\infl \P \geq -n$.
Let $\f^{'}:\P^{'}\ra \X$ be another $\lambda$-pure projective resolution of $\X$.
By the definition, we have a quasi-isomorphism
$$\Hom_{\KA}(\P,\f^{'}): \Hom_{\KA}(\P,\P^{'})\ra \Hom_{\KA}(\P,\X)
$$
which means that there exists a cochain map $\g:\P\ra \P^{'}$ such that $\f=\f^{'}\g$. By Proposition \ref{prop206}, we know $\g$ is also a $\lambda$-pure quasi-isomorphism.
Then $\g$ is a homotopy equivalence by \cite[Corollary 4.5]{Wang}, and we have the $\lambda$-pure acyclic complex
$$
\xymatrix@C=0.5cm{
  \cdots \ar[r] & P^{' -n-1} \ar[r]^{} &  P^{' -n} \ar[r]^{} & \coker d_{\P^{'}}^{-n-1} \ar[r] & 0 }
$$
is contractible by direct checking. This implies that $\coker d_{\P^{'}}^{-n-1}$ is a $\lambda$-pure projective object in $\cal A$.

(2) $\Ra$ (3)\ \ Let $\P^{'}\ra \X$ be a $\lambda$-pure projective resolution of $\X$.
The $$
\xymatrix@C=0.5cm{
  \cdots \ar[r] & P^{' -n-1} \ar[r]^{} &  P^{' -n} \ar[r]^{} & \coker d_{\P^{'}}^{-n-1} \ar[r] & 0 }
$$
is $\lambda$-pure acyclic since $\infl \P^{'} =\infl \X \geq -n$.
Moreover, the complex is split by the assumption of which $\coker d_{\P^{'}}^{-n-1}$ is $\lambda$-pure projective.
So there exists an object $Q\in \cal A$, such that $P^{' -n}=\coker d_{\P^{'}}^{-n-1}\bigoplus Q$.
Clearly, $Q$ is also a $\lambda$-pure projective object.
We set
 $$
\xymatrix@C=0.5cm{
 \P_1:=  &\cdots \ar[r] & 0 \ar[r]^{} & \coker d_{\P^{'}}^{-n-1} \ar[r]^{} &  P^{' -n+1} \ar[r]^{} &  P^{' -n+2}\ar[r] & \cdots }
$$
and
 $$
\xymatrix@C=0.5cm{
 \P_2:=  & \cdots \ar[r] & P^{' -n-2} \ar[r]^{} &  P^{' -n-1} \ar[r]^{} & Q \ar[r] & 0\ar[r] & \cdots . }
$$
Then $\P^{'}=\P_1\bigoplus \P_2$, and it satisfies the conditions of assertion (3).

(3) $\Ra$ (1)\ \ There exists an embedding $\P_1\ra \P^{'}$ by the assumption, clearly, it is a $\lambda$-pure quasi-isomorphism.
Then the composition $\P_1\ra \P^{'}\ra \X$ is also a $\lambda$-pure quasi-isomorphism, which satisfies $P_1^i=0$ for any $i<-n$.
This implies that $\plpdA\X\leq n$ as wanted.
\end{proof}

\begin{corollary}
For any object $X\in \cal A$, the following statements are equivalent.

(1)\ \ $\plpdA X\leq n$.

(2)\  \ For any $\lambda$-pure exact sequence $0\ra K\ra P_{n-1}\ra \cdots \ra P_1\ra X\ra 0$ with $P_i\in \lPP$, we have $K\in \lPP$.
\end{corollary}

\begin{theorem}\label{theorem global dimension}
For any $n\in \mathbb{Z}$, the following statements are equivalent.

(1)\ \ $\plgldA\leq n$;

(2)\ \ $\plpdA\X\leq n-\infl\X$ for any complex $\X\in \D_{\lambda}^{b}(\cal A)$;

(3)\ \ $\plext_{\cal A}^{i}(\textbf{X},\textbf{Y})=0$ for any $\X, \Y\in \D_{\lambda}^{b}(\cal A)$, and $i> n+\supl\Y-\infl \X$.
\end{theorem}

\begin{proof}
(1) $\Rightarrow$ (2). Let $\infl\X=-m$ and $\P\ra \X$ be a $\lambda$-pure projective resolution of $\X$.
Then $\infl\P=-m$ and hence $\P$ is $\lambda$-pure acyclic in degree $< -m$.
So
$$\cdots \ra P^{-m-1}\ra P^{-m} \ra \coker d_{\P}^{-m-1}\ra 0
$$
is a $\lambda$-pure projective resolution of $\coker d_{\P}^{-m-1}$.
By the assumption, we have $\plpdA \coker d_{\P}^{-m-1}\leq n$.
Let
$$0\ra Q^{-n}\ra \cdots \ra Q^{-1}\ra Q^0 \ra \coker d_{\P}^{-m-1}\ra 0
$$
be a $\lambda$-pure projective resolution of $\coker d_{\P}^{-m-1}$.
Then 
we know
$$\cdots \ra 0 \ra Q^{-n}\ra \cdots \ra Q^{-1}\ra Q^0 \ra P^{-m+1}\ra P^{-m+2}\ra \cdots
$$
is a $\lambda$-pure projective resolution of $\X$. Notice $Q^{-n}$ is in degree $-m-n$, thus $\plpdA\X \leq m+n=n-\infl\X$ as wanted.

(2) $\Rightarrow$ (3). Set $\supl\Y=r$. Then we know $\Y$ is $\lambda$-pure acyclic in degree $>r$.
By the proof of Lemma \ref{lem207}(2), we have a $\lambda$-pure quasi-isomorphism $\Y^{'}\ra \Y$,
where $\Y^{'}$ is the right canonical truncation complex of $\Y$ at degree $r$.
For any $\X\in \DlAs b$, let $\P\ra \X$ be a $\lambda$-pure projective resolution of $\X$ with $P^i=0$ for all $i<-(n-\infl X)$.

Hence, we have
$\plext_{\cal A}^{i}(\textbf{X},\textbf{Y})= \H^{i}\Hom_{\cal A}(\textbf{P},\textbf{Y})\cong
\H^{i}\Hom_{\cal A}(\textbf{P},\textbf{Y}^{'})=0$
for $i>n-\infl X+r$ as expected.

(3) $\Rightarrow$ (1). Just let $\X, \Y$ be $\cal A$-objects.
\end{proof}

\section{The relations between derived categories and $\lambda$-pure derived categories}

Let $\K_{ac}{(\cal A)}$ be the subcategory of $\K(\cal A)$ consisting of all acyclic cochain complexes. Define the corresponding derived category $\D{(\cal A)}:=\K{(\cal A)}/\K_{ac}{(\cal A)}$.
Similarly, $\D^{\ast}{(\cal A)}:=\K^{\ast}{(\cal A)}/\K^{\ast}_{ac}{(\cal A)}$, where $\ast \in \{b, +,-\}$.
In this section, we explore the relations between bounded (above) derived categories and bounded (above) $\lambda$-pure derived categories.

In order to make the structure of bounded (above) derived categories simple, we always suppose that the Grothendieck category $\cal A$ has enough projective objects, and denote by $\cal P$ the full subcategory consisting of all projective objects.

\begin{example}
The following categories are Grothendieck categories, and all of them have enough projective objects.

(i)\ \ The right module category Mod-$R$ of an arbitrary unitary ring $R$;

(ii)\ \ The category of all representations of a quiver $\cal Q$, denoted by Rep$({\cal Q}, {\Lambda})$, in Mod-$\Lambda$ (see \cite[\S 7.1]{J2016}), where Mod-$\Lambda$ is the right module category of an arbitrary unitary ring $\Lambda$;

(iii)\ \ The right $C$-comodule category ${\cal M}^C$ of a semiperfect coalgebra $C$ (see \cite[Theorem 10]{Lin1977}).
\end{example}

In this case, the relative results about the derived category of abelian category, see for example \cite[Section 5]{ZP2015}, also hold in the Grothendieck condition.
It is well-known that there exist the triangle equivalences $\D^{b}{(\cal A)}\simeq \K^{-,b}{(\cal P)}$ and $\D^{-}{(\cal A)}\simeq \K^{-}{(\cal P)}$.

We will use the following Lemma.
\begin{lemma}\label{lemma402}\cite[corollary 4-3]{Verdier1977}
Suppose that ${\cal D}_1$ and ${\cal D}_2$ are triangulated subcategories of a triangulated category $\cal C$, and ${\cal D}_1$ is a subcategory of ${\cal D}_2$.
Then ${\cal D}_2/{\cal D}_1$ is a triangulated subcategory of ${\cal C}/{\cal D}_1$, and there is a triangle equivalence $({\cal C}/{\cal D}_1)/({\cal D}_2/{\cal D}_1)\simeq {\cal C}/{\cal D}_2$.
\end{lemma}

By \cite[Theorem 4.9]{Wang}, we have a triangle equivalence
$\K^-(\lPP)\simeq  \D_{\lambda}^-(\cal A).$
Next we will restrict it to $\K_{ac}^{-}(\lPP)$ and also obtain a triangle equivalence.

\begin{theorem}
Let $\cal A$ be a Grothendieck category with enough projective objects.
Then there exists the following equivalence.
$$\K_{ac}^{-}(\lPP) \simeq \K_{ac}^{-}({\cal A})/\K_{\lambda}^{-}({\cal A}).$$
Moreover, we have:
$$\D^-({\cal A})\simeq \K^-(\lPP)/\K_{ac}^{-}({\lPP})\simeq \D_{\lambda}^-({\cal A})/\K_{ac}^{-}({\lPP}).
$$
\end{theorem}

\begin{proof}
Let $F^{'}:\K_{ac}^{-}(\lPP) \ra \K_{ac}^{-}({\cal A})/\K_{\lambda}^{-}({\cal A})$
be the composition of the embedding functor
$\K_{ac}^{-}(\lPP) \hookrightarrow \K_{ac}^{-}({\cal A})$
and the localization functor
$Q:\K_{ac}^{-}({\cal A}) \ra \K_{ac}^{-}({\cal A})/\K_{\lambda}^{-}({\cal A})$.
Consider the following diagram
$$\xymatrix@C=0.5cm{
  \K_{ac}^{-}(\lPP)  \ar[r]^{inc} \ar[d]_{F^{'}}
  & \K^{-}(\lPP) \ar[d]_{\cong}^{F}
  \\
 \K_{ac}^{-}({\cal A})/\K_{\lambda}^{-}({\cal A})
 \ar[r]^{inc} &
 \K^-{(\cal A)}/\K_{\lambda}^{-}({\cal A})
  }
$$
By \cite[lemma 4.8]{Wang}, we have $F^{'}$ is fully faithful.
For any complex $\X\in \K_{ac}^{-}(\cal A)$, there exists a $\lambda$-pure quasi-isomorphism $\P\ra \X$ with $\P\in \K^{-}(\lPP)$.
Since a $\lambda$-pure quasi-isomorphism is also a quasi-isomorphism and $\X$ is acyclic, we know that $\P$ is also acyclic, that is, $\P\in \K_{ac}^{-}(\lPP)$.
Therefore, $F^{'}$ is dense as expected.

In this case,
\begin{align*}\K^{-}(\lPP)/\K_{ac}^{-}(\lPP)&\simeq
(\K^-{(\cal A)}/\K_{\lambda}^{-}({\cal A}))/
(\K_{ac}^{-}({\cal A})/\K_{\lambda}^{-}({\cal A}))\\
&\simeq \K^-{(\cal A)}/ \K_{ac}^{-}({\cal A})\\
&=
\D^-(\cal A).
\end{align*}
Note that the second equivalence holds by Lemma \ref{lemma402}.
\end{proof}

Let $\K_{ac}^{b}(\lPP)$ be the subcategory of $\K^b(\cal A)$ consisting of all acyclic cochain complexes of $\lambda$-pure projective objects.
The following theorem is our main result in this section.

\begin{theorem}\label{th401}
If $\plgldA$ is finite, then there are triangle equivalences
$$\D^{b}({\cal A})
\simeq \K_{\lambda}^{-,b}(\lPP)/ \K_{ac}^{b}(\lPP)
\simeq \D_{\lambda}^{b}{(\cal A)}/\K_{ac}^{b}(\lPP) .
$$
\end{theorem}

In order to complete the proof of Theorem \ref{th401}, we need the following preparations.

\begin{lemma}\label{lem405}
Let $\X\in \K^{-, b}({\cal P})$. Then there exists a quasi-isomorphism $\X\ra \P$ with $\P\in \K_{\lambda}^{-, b}(\lPP)$.
\end{lemma}

\begin{proof}
Let $\X\in \K^{-, b}({\cal P})$, i.e., there exists an integer $n$, such that
$$\xymatrix@C=0.5cm{
  \X_1:& \cdots \ar[r] & X^{n-2} \ar[r] & X^{n-1}\ar[r]^{} & \ker d_X^{n}\ar[r] & 0 }$$
is acyclic.
Assume $$\xymatrix@C=0.5cm{
  \Q: &\cdots \ar[r] & Q^{n-2}  \ar[r] & Q^{n-1} \ar[r]^{} & \ker d_X^{n} \ar[r] &0}
$$
is a $\lambda$-pure projective resolution of $\ker d_{\X}^{n}$.
By the usual Comparison Theorem, there exists a cochain map $\g: \X_1\ra \Q$.
Then we construct the complex $\P$ by set $P^i=X^i$ for $i\geq n$ and $P^j=Q^j$ for $i< n$.
So $\P\in \K_{\lambda}^{-, b}(\lPP)$, and $\g$ induces a cochain map $\f$ from $\X$ to $\P$.
Since both $(\X_1)_{\leq n-1}$ and $\Q_{\leq n-1}$ are acyclic complexes, we have $\Cone(\g_{\leq n-1})$ is also acyclic by \cite[Proposition 3.8]{Wang}.
Thus, $\g_{\leq n-1}$ is a quasi-isomorphism.
This means $\f: \X\ra \P$ is a quasi-isomorphism as wanted.
\end{proof}

\begin{lemma}\label{lem406}
For any cochain map $\g: \X\ra \Q$ with $\X\in \K^{-, b}({\cal P})$ and $\Q \in \K_{\lambda}^{-, b}(\lPP)$. There exist a quasi-isomorphism $\f: \X\ra \P$ with $\P \in \K_{\lambda}^{-, b}(\lPP)$ and a cochain map $\h: \P\ra \Q$, such that $\g= \h\f$ in $\K^{-}(\cal A)$.
\end{lemma}
\begin{proof}
For  $\X\in \K^{-, b}({\cal P})$, $\Q \in \K_{\lambda}^{-, b}(\lPP)$, there exist integers $n_1$ and $n_2$, such that $\X$ is acyclic in degree $\leq n_1$ and $\Q$ is $\lambda$-pure acyclic in degree $\leq n_2$.
Set $n=\min\{n_1,n_2\}$.
By the construction of the above Lemma, there exists a quasi-isomorphism $\f: \X \ra \P$ with $X^i=P^i$ for all $i\geq n$ and $\P\in \K_{\lambda}^{-, b}(\lPP)$.

Now we construct $\h:\P\ra \Q$ by induction. For $i\geq n$, set $f^i=\Id_{\X}^{i}$ and $h^i=g^i$.
For $i\leq n-1$, we consider the following diagram
$$\xymatrix@C=0.38cm{
\cdots \ar[r]
&
  X^{n-2}\ar[dr]^{f^{n-2}} \ar[rr] \ar[dd]^(0.6){g^{n-2}}
&&
  X^{n-1}\ar[rr]^{}\ar[dr]^{f^{n-1}} \ar[dd]^(0.6){g^{n-1}}
&&
  X^{n}\ar[rr]\ar[dr]^{f^n:=\Id_{\X}^{n}} \ar[dd]^(0.6){g^n}
&&
  X^{n+1}\ar[r]\ar[dr]^{f^{n+1}:=\Id_{\X}^{n+1}}\ar[dd]^(0.6){g^{n+1}}
&
\cdots
&
\\
&
\cdots \ar[r]
&
  P^{n-2} \ar[rr]^(0.4){d_{\P}^{n-2}}\ar@{-->}[dl]^{h^{n-2}}
&&
  P^{n-1}\ar[rr]^(0.4){d_{\P}^{n-1}}\ar@{-->}[dl]^{h^{n-1}}
&&
  P^{n}\ar[rr]^(0.4){d_{\P}^{n}}\ar@{-->}[dl]^{h^n:=g^n}
&&
  P^{n+1}\ar[r]\ar@{-->}[dl]^{h^{n+1}:=g^{n+1}}
&
  \cdots
\\
\cdots \ar[r]
&
  Q^{n-2} \ar[rr]_{d_{\Q}^{n-2}} \ar@{->>}[dr]_{r^{n-2}}
&&
  Q^{n-1}\ar[rr]_{d_{\Q}^{n-1}} \ar@{->>}[dr]_{r^{n-1}}
&&
  Q^{n}\ar[rr]_{d_{\Q}^{n} }
&&
  Q^{n+1}\ar[r]
&
  \cdots
&
\\
&&\ker d_{\Q}^{n-1} \ar@{ >->}[ur]_{t^{n-1}}
&&
\ker d_{\Q}^{n} \ar@{ >->}[ur]_{t^{n}}
&&&&&
}$$
Since $d_{\Q}^{n}\cdot h^{n}d_{\P}^{n-1}
=h^{n+1}d_{\P}^{n}d_{\P}^{n-1}=0$,
we have $h^{n}d_{\P}^{n-1} = t^n \varphi^{n-1}$, where $\varphi^{n-1}:P^{n-1}\ra \ker d_{\Q}^{n}$.
Moreover, $\varphi^{n-1}$ can be lifted to $Q^{n-1}$ since $r^{n-1}$ is a $\lambda$-pure epimorphism and $P^{n-1}\in \lPP$.
That is, there exists a morphism $h^{n-1}: P^{n-1}\ra Q^{n-1}$, such that $\varphi^{n-1}=r^{n-1}h^{n-1}$.
Hence $h^nd_{\P}^{n-1}=t^n r^{n-1}h^{n-1}=d_{\Q}^{n-1}h^{n-1}$.
Inductively, we have $h^id_{\P}^{i-1}=d_{\Q}^{i-1}h^{i-1}$ for all $i\leq n-1$.
This implies that we get the cochain map $\h:\P\ra \Q$.

Next we check that the above $\h$ satisfies $\g= \h\f$ in $\K^{-}(\cal A)$.
For $i\geq n$, set $s^{i}=0$.
In this case, $g^{i}-h^if^i=0=d_{\Q}^{i-1}s^i+s^{i+1}d_{\X}^{i}$.
For $i\leq n-1$, we construct $s^i$ by induction.
Since $d_{\Q}^{n-1}(g^{n-1}-h^{n-1}f^{n-1})
=g^nd_{\X}^{n-1}-h^nd_{\P}^{n-1}f^{n-1}
=h^nd_{\X}^{n-1}-h^nd_{\X}^{n-1}=0$,
we have $g^{n-1}-h^{n-1}f^{n-1}=t^{n-1}\psi^{n-1}$, where $\psi^{n-1}:X^{n-1}\ra \ker d_{\Q}^{n-1}$.
Moreover, $\psi^{n-1}$ can be lifted to $Q^{n-2}$, that is, there exists a morphism $s^{n-1}: X^{n-1}\ra Q^{n-2}$, such that $\psi^{n-1}=r^{n-2}s^{n-1}$.
Hence we have
\begin{align*}
g^{n-1}-h^{n-1}f^{n-1}&=t^{n-1}r^{n-2}s^{n-1}\\
&=d_{\Q}^{n-2}s^{n-1}\\
&=d_{\Q}^{n-2}s^{n-1}+s^{n}d_{\X}^{n-1}.
\end{align*}

Furthermore, since
\begin{align*}
d_{\Q}^{n-2}s^{n-1}d_{\X}^{n-2}&=(g^{n-1}-h^{n-1}f^{n-1})d_{\X}^{n-2}\\
&=d_{\Q}^{n-2}g^{n-2}-h^{n-1}(d_{\P}^{n-2}f^{n-2})\\
&=d_{\Q}^{n-2}g^{n-2}-d_{\Q}^{n-2}h^{n-2}f^{n-2},
\end{align*}
we have $g^{n-2}-h^{n-2}f^{n-2}-s^{n-1}d_{\X}^{n-2}=d_{\Q}^{n-3}s^{n-2}$ with $s^{n-2}:X^{n-2}\ra Q^{n-3}$.
Therefore, we complete the construction of $\s$ by induction, and $\s:\g\sim \h\f$ as wanted.
\end{proof}

\begin{lemma}\label{lem404}
For any $\X \in \K_{\lambda}^{-,b}(\lPP)\simeq \D_{\lambda}^{b}(\cal A)$, if $\plgldA$ is finite, then $\X\in \K^{b}(\lPP)$.
\end{lemma}

\begin{proof}
It is clear by a truncation.
\end{proof}

Next we shall prove Theorem \ref{th401}.

\begin{proof}
Let $G$ be the composition functor $\K_{\lambda}^{-,b}(\lPP) \ra \K^{-}({\cal A})\ra \D^{-}(\cal A)$, where the first functor is an embedding functor and the second one is the Verdier quotient.
Then we know $G(\K_{ac}^{b}(\lPP))=0$ since $\K_{ac}^{b}(\lPP) \subseteq \K_{ac}^{-}(\cal A)$.
By the universal property of quotient functor, there exists a unique triangle functor $\overline{G}: \K_{\lambda}^{-,b}(\lPP)/ \K_{ac}^{b}(\lPP) \ra \D^{-}(\cal A)$.
Clearly, $\Im(\overline{G})\subseteq \D^{b}(\cal A)$.

For any $\X\in \K^{-,b}({\cal P})\simeq \D^{b}(\cal A)$, by Lemma \ref{lem405}, there exists $\P\in \K_{\lambda}^{-,b}(\lPP)$ such that $\overline{G}(\P)=\X$, that is, $\overline{G}$ is dense.

Let $\P_1, \P_2\in \K_{\lambda}^{-,b}(\lPP)$ and $\xymatrix@C=0.5cm{
  \P_1 \ar@{<=}[r]^{~\s} & \Z  \ar[r]^{\a} & \P_2}$ be a morphism in $\D^{b}(\cal A)$ with $\Z\in \K^{b}(\cal A)$, and $\s$ is a quasi-isomorphism which satisfies $\Cone(s) \in \K_{ac}^{b}(\lPP)$.
Then there exists a quasi-isomorphism $\uu:\P\ra \Z$ with $\P\in \K^{-,b}(\cal P)$ by \cite[Theorem 4.2.1]{ZP2015}.
Consider the morphisms $\s\uu: \P\ra \P_1$ and $\a\uu:\P\ra \P_2$, it follows from Lemma \ref{lem406} that there exist a quasi-isomorphism $\vv:\P\ra \Q$ with $\Q\in \K_{\lambda}^{-,b}(\lPP)$, and cochain maps $\t:\Q\ra \P_1$ and $\b: \Q\ra \P_2$, such that $\s\uu=\t\vv$ and $\a\uu=\b\vv$ in $\K^-(\cal A)$.
Then the following diagram
$$\xymatrix{
  & \Z \ar[dr]^{\a}  \ar@{=>}[dl]_{\s}    \\
\P_1 \ar@{<=}[r]^{~~\s\uu} & \P \ar@{=>}[u]_{\uu}\ar@{=>}[d]^{\vv} \ar[r]^{\a\uu~~} & \P_2,      \\
 &  \Q\ar@{=>}[ul]^{\t} \ar[ur]_{\b}   }
$$
commutes in $\K^-(\cal A)$, where the double arrows denote quasi-isomorphisms.
Moreover, $\t$ is also a quasi-isomorphism since both $\s\uu$ and $\vv$ are quasi-isomorphisms.
Then $\Cone(\t)$ is acyclic.
Since both $\Q$ and $\P_1$ lie in $\K_{\lambda}^{-,b}(\lPP)$, as a corollary of \cite[Proposition 3.8]{Wang}, we know $\Cone(\t)\in \K_{\lambda}^{-,b}(\lPP)$ .
It follows from Lemma \ref{lem404} that $\Cone(\t)\in \K_{ac}^{b}(\lPP)$.
Hence, the right fraction $$\b/\t\in \Hom_{\K_{\lambda}^{-,b}(\lPP)/ \K_{ac}^{b}(\lPP)}(\P_1, \P_2),$$
and $\overline{G}(\b/\t)=\b/\t=\a/\s$. This means that $\overline{G}$ is full.

At last, by \cite[page 446]{Rickard1989}, also see \cite[Proposition 1.5.2]{ZP2015}, the full triangle functor $\overline{G}$ is faithful so long as it does not take any non-zero objects to zero.
Assume $\P\in \K_{\lambda}^{-,b}(\lPP)$ and $\overline{G}(\P)=0$.
Then $\P$ is acyclic by the definition of $\overline{G}$.
Hence, $\P\in \K_{ac}^{b}(\lPP)$ by Lemma \ref{lem404},
and we complete the proof.
\end{proof}

\section{$\lambda$-pure singularity categories}

Let $\cal A$ be an abelian category with enough projective objects,
and $\cal P$ the full subcategory of $\cal A$ consisting of all projective objects.
The ordinary singularity category of $\cal A$ is denoted by a Verdier quotient category $\D_{sg}^{b}{(\cal A)}:=\D^{b}{(\cal A)}/\K^{b}(\cal P)$.
Similarly, we can define the $\lambda$-pure version of singularity category as follows.

\begin{definition}
The $\lambda$-pure singularity category is defined to be the Verdier quotient
$$\DlsgA:=\D_{\lambda}^{b}{(\cal A)} / \K^{b}(\lPP)\simeq \K_{\lambda}^{-,b}{(\lPP)} / \K^{b}(\lPP).$$
\end{definition}

\begin{proposition}
The $\lambda$-pure singularity category $\DlsgA=0$ if and only if the $\lambda$-pure global dimension $\plgldA$ is finite.
\end{proposition}

\begin{proof}
By Lemma \ref{lem404}, we have $\D_{\lambda}^{b}{(\cal A)}\simeq \K^{b}(\lPP)$ when $\plgldA$ is finite, that means $\DlsgA=0$.

Suppose $\DlsgA=0$. Let $M\in \cal A$, and $M=0$ in $\DlsgA$.
Then there exists some $\X\in \K^{b}(\lPP)$ such that $M\cong \X$ in $\D_{\lambda}^{b}{(\cal A)}$.
Denote this isomorphism by a right fraction
$\xymatrix@C=0.5cm{
  M \ar@{<=}[r]^{~f} & \Y  \ar[r]^{\alpha} & \X}$,
where $f$ and $\alpha$ are $\lambda$-pure quasi-isomorphisms.
By \cite[Corollary 4.4]{Wang}, there exists a $\lambda$-pure quasi-isomorphism $\beta: \X\ra \Y$ such that $\alpha\beta\sim \Id_{\X}$.
So $f\beta:\X\ra M$ is also a $\lambda$-pure quasi-isomorphism, and hence $\X$ is a $\lambda$-pure acyclic complex in degree $\neq 0$.
Consider the truncation
$$\xymatrix@C=0.5cm{
 \tau_{\leq 0}\X  & \cdots \ar[r] & X^{-2} \ar[r] & X^{-1} \ar[r] & \ker d_{\X}^{0} \ar[r] & 0 }
$$
Then the composition $\tau_{\leq 0}\X\ra \X\ra M$ is a $\lambda$-pure quasi-isomorphism.
Since $\X\in \K^{b}(\lPP)$, we assume $X^i=0$ for all $i>n$.
Then we have a $\lambda$-pure acyclic complex $0\ra \ker d_{\X}^0 \ra X^0\ra \cdots \ra X^n \ra 0$ with $X^{i}\in \lPP$.
Hence it is contractible by Proposition \ref{contractible}.
Then $\ker d_{\X}^0 \in \lPP$ since $\lPP$ is closed under direct summands.
That means $\tau_{\leq 0}\X\ra M$ is a bounded $\lambda$-pure projective resolution, and then $\plgldA<\infty$.
\end{proof}

There is a classic result about the general singularity category, i.e., Buchweitz-Happel Theorem. Next, we want to explore whether there is a $\lambda$-pure version of  Buchweitz-Happel Theorem. We briefly describe the classic Buchweitz-Happel Theorem at first.

Let $\cal A$ be an abelian category with enough projective objects, and $\cal P$ is the class of all projective objects.
An object $X\in \cal A$ is called Gorenstein projective if there exists an exact complex
$$\P: \cdots\ra P^{-1} \ra P^0 \ra P^{1} \ra \cdots$$
with $P^i\in \cal P$ for all integers $i$, such that $X \cong \Im(P^{-1}\ra P^0)$
and $\Hom_{\cal A}(Q,\P)$ is acyclic for any $Q\in \cal P$.
Denote by $\cal{G(A)}$ the full subcategory consisting of all Gorenstein projective objects in $\cal A$.
Then it is a Frobenius category with relative projective-injective objects being projective objects. 
Hence, the stable category $\underline{\cal{G(A)}}$ of $\cal{G(A)}$ modulo $\cal P$ is a triangulated category by Happel \cite{Happel1987, Happel1988}.
Then there exists the following result, which is called Buchweitz-Happel Theorem and its inverse.

\begin{theorem}\label{theoremB-H}\cite[4.4.1]{Buchweitz}\cite{Happel1991}
Let $\cal A$ be an abelian category with enough projective objects.
Then the canonical functor $$F: \underline{\cal{G(A)}} \ra \D_{sg}^{b}(\cal A)$$ is a fully faithful triangle functor.
Moreover, $F$ is a triangulated equivalent if and only if the Gorenstein projective dimension of each object in $\cal A$ is finite, where the ``only if" part holds by Zhu\cite{zhu2012}.
\end{theorem}

\begin{remark}
Chen\cite{Chen2011}, and Li-Huang\cite{Lihuanhuan2015, Lihuanhuan2020} generalized the above Theorem, and they got the relative version of Buchweitz-Happel Theorem, respectively.
Moreover, Gao-Zhang considered the Gorenstein version in \cite{GZ2010}.
And  Bao \emph{et.al.}\cite{Bao2015} introduced the Gorenstein version.
Cao \emph{et.al.} \cite{Cao2019} explored the pure version.
\end{remark}

In order to consider the $\lambda$-pure version of Buchweitz-Happel Theorem,
we define a Gorenstein category with respect to $\lambda$-pure projective objects, denoted by $\GlPP$, that is, $M\in \GlPP$, if there exists an exact complex $\cdots\ra P^{-1} \ra P^0 \ra P^{1} \ra \cdots$ with every $P^i \in \lPP$, such that $M \cong \Im(P^{-1}\ra P^0)$ and the above sequence is still exact after applying $\Hom_{\cal A}(\lPP,-)$ and $\Hom_{\cal A}(-,\lPP)$, respectively.
However, in this case, we immediately get the following result, which directly leads to that the general construction of Buchweitz-Happel Theorem is invalid for the $\lambda$-pure version.

\begin{proposition}
Let $\cal A$ be a Grothendieck category, then $\GlPP= \lPP$.
\end{proposition}

\begin{proof}
$\lPP\subseteq \GlPP$ is clear.
Conversely, let $M\in \lPP$, then the sequence $\cdots\ra P^{-1} \ra P^0 \ra P^{1} \ra \cdots$ in the definition of $\GlPP$ is in fact $\lambda$-pure acyclic, hence it is contractible by Proposition \ref{prop204}.
And we get $M$ is a direct summand of $P_0$, so it is a $\lambda$-pure projective object.
\end{proof}

In homological algebra, projective and injective are often dual, we know that if $\cal A$ is an abelian category with enough injective objects, then the classical derived categories $\DA$ can be  represented by injective objects.
As for $\lambda$-pure derived category $\DlA$, if the underlying category $\cal A$ is a Grothendieck category with enough $\lambda$-pure injective objects, then $\DlA$ also can be represented by $\lambda$-pure injective objects.
This is dual to $\lambda$-pure projective case in \cite{Wang}, we omit relevant description.
In fact, the previous claim about ``with enough $\lambda$-pure injective objects" is very complicated.
Rosick\'{y} \cite[Question 1]{Rosicky} set a question: Does there exist a regular cardinal $\lambda$ such that a locally presentable additive category has enough $\lambda$-pure injectives?
There seems to be limited research about it at present, and we only find Izurdiaga-\u{S}aroch \cite[Theorem 6.6]{Izurdiaga2021} proved that a ring $R$ is right pure-semisimple if and only if there exists a regular uncountable $\lambda$ such that the category Mod-$R$ has enough $\lambda$-pure injective objects.

If we roughly assume that the answer to Rosick\'{y} \cite[Question 1]{Rosicky} is affirmative, i.e., the Grothendieck category, which is locally presentable by \cite[page 2]{Rosicky}, has enough $\lambda$-pure injective objects.
Whether the $\lambda$-pure version of Buchweitz-Happel Theorem can be verified by $\lambda$-pure injective objects?
This is still uncertain, it depends entirely on a conjecture: A chain complex $\bm{F}$ is $\lambda$-pure acyclic if and only if any chain map $\bm{F}\ra \bm{E}$ is null-homotopic, where $\bm{E}$ is a complex with $\lambda$-pure injective objects.
Note this conjecture is the $\lambda$-pure version of \cite[Conjecture 3.4 ]{Daniel2017}.
If the conjecture is negative, then we may construct the $\lambda$-pure version of Buchweitz-Happel Theorem as is in the classic case, which is also similar to Chen and Li-Huang.
Else if the conjecture is positive, then it is dual to the $\lambda$-pure projective case. Therefore the general method is not feasible too.

~\\
\textbf{Author contributions}. All the co-authors have contributed equally in all aspects of the preparation of this submission.
~\\
\textbf{Conflict of interest statement.} The authors declare that they have no known competing financial interests or personal relationships that could have appeared to influence the work reported in this paper.
~\\
\textbf{Funding.} This work is supported by National Natural Science Foundation of China(12071120).
~\\
\textbf{Data availability.} No data was used for the research described in the article.


\begin{thebibliography}{10}


\bibitem{Adamek1994} J. Ad\'{a}mek, J. Rosicky. Locally presentable and accessible categories. Cambridge University Press, 1994.

\bibitem{Adamek2004} J. Ad\'{a}mek, J. Rosicky. On pure quotients and pure subobjects. Czechoslovak Math. J., 54(2004), no.3, 623-636.


\bibitem{J2016} J. Asadollahi, P. Bahiraei, R. Hafezi, R. Vahed, On relative derived categories. Comm. Algebra, 44(2016), 5454-5477.





\bibitem{Bao2015} Y. Bao, X. Du, Z. Zhao, Gorenstein singularity categories, J. Algebra, 428(2015), 122-137.

\bibitem{Buchweitz} R. Buchweitz. Maximal Cohen-Macaulay modules and tate cohomology over Gorenstein rings. unpublish manuscript. 1987.

\bibitem{Theo2010} T. B\"{u}hler. Exact categories. Expo. Math., 28(2010), no.1, 1-69.

\bibitem{Cao2019} T. Cao, Z. Liu, X. Yang. Buchweitz Theorem in pure singularity category (in Chinese). Acta Math. Sinica (Chinese Ser.), 62(2019), no.4, 552-560.


\bibitem{Chen2011} X. Chen. Relative singularity categories and Gorenstein-projective modules. Math. Nachr., 284(2011), no.2-3, 199-212.


\bibitem{Christensen2023} L. Christensen, N. Ding, S. Estrada, J. Hu, H. Li, P. Thompson. The singularity category of an exact category applied to characterize Gorenstein schemes. Q.J. Math., 74(2023), no.1, 1-27.




\bibitem{GZ2010} N. Gao, P. Zhang. Gorenstein derived categories and Gorenstein singularity categories. unpublish manuscript. 2010.


\bibitem{Gelfand2003}S. Gelfand, Y. Manin, Methods of homological algebra(second edition). Springer Monographs in Mathematics, Springer-Verlag, Berlin, 2003.

\bibitem{Gillespie2016} J. Gillespie. The derived category with respect to a generator. Ann. Mat. Pura Appl.(4), 195(2016), no.2, 371-402.

\bibitem{Happel1987} D. Happel. On the derived category of a finite-dimensional algebra. Comment. Math. Helv., 62(1987), no.3,  339-389.

\bibitem{Happel1988} D. Happel. Triangulated categories in the representation theory of finite-dimensional algebras. London Mathematical Society Lecture Note Series, 119. Cambridge University Press, Cambridge. 1988.

\bibitem{Happel1991} D. Happel. On Gorenstein algebras. Progr. Math., 95(1991), 389-404.




\bibitem{Huang2013} Z. Huang. Proper resolutions and Gorenstein categories. J. Algebra, 393(2013), 142-169.

\bibitem{Izurdiaga2021} M. Izurdiaga, J. \u{S}aroch. Module class induced by complexes and $\lambda$-pure injective modules. arXiv: 2104.08602v1[math.RT]17 Apr2021.

\bibitem{Keller1990} B. Keller, Chain complexes and stable categories. Manuscripta Math., 67 (1990), no.4, 379-417.


\bibitem{Lihuanhuan2015} H. Li, Z. Huang. Relative singularity categories. J. Pure Appl. Algebra, 219(2015), no.9, 4090-4104.

\bibitem{Lihuanhuan2020} H. Li, Z. Huang. Relative singularity categories II. Kodai Math. J., 43(2020), no.3, 431-453.


\bibitem{Lin1977} B. Lin. Semiperfect coalgebras. J. Algebra, 49(1977), no.2, 357-373.




\bibitem{Neeman1990} A. Neeman. The derived category of an exact category. J. Algebra, 135(1990), no.2, 388-394.

\bibitem{Rickard1989} J. Rickard. Morita theory for derived categories. J. London Math. Soc.(2), 39(1989), no.3, 436-456.

\bibitem{Rosicky} J. Rosick\'{y}. Generalized purity, definablity and Brown representability. Some
Trends in Algebra. Prague, 2009.

\bibitem{Daniel2017}D. Simson. Flat Complexes, Pure periodicity and pure acyclic complexes. J. Algebra, 480(2017), 298-308.

\bibitem{Shen2022} L. Shen, N. Ding, M. Wang. A solution to the conjecture on $\lambda$-pure acyclic complexes. J. Algebra, 605(2022), 58-73.



\bibitem{Verdier1977} J. Verdier. Cat\'{e}gories d\'{e}riv\'{e}es: quelques r\'{e}sultats (\'{e}tat 0). (French). Lecture Notes in Math., vol.569, Springer, Berlin, 1977, 262-311.


\bibitem{Wang} X. Wang, H. Yao, L. Shen. $\lambda$-pure derived categories of a Grothendieck category. J. Algebra Appl., (23)2024, no.12: 2450197, 22pages.

\bibitem{ZP2015} P. Zhang. Trangulated categories and derived Categories (in Chinese). Science Press, BeiJing. 2015.

\bibitem{HZY2016} Y. Zheng, Z. Huang. On pure derived categories. J. Algebra, 454(2016), 252-272.

\bibitem{zhu2012} S. Zhu. Left homotopy theory and the inverse of Buchweitz's theorem. Thesis(Master), Shanghai Jiaotong University. 2012.


\end{thebibliography}
\end{document}